\documentclass[12pt, leqno]{amsart}
\usepackage[all]{xy}
\usepackage{amsfonts,amsmath,oldgerm,amssymb,amscd}
\UseComputerModernTips
\newtheorem{theorem}{Theorem}

\newtheorem{lemma}[subsection]{{\bf Lemma}}
\newtheorem{coro}[subsection]{{\bf Corollary}}

\newcommand{\al}{\alpha}

\newcommand{\om}{\omega}
\newcommand{\del}{\delta}

\newcommand{\Z}{\mbox{$\mathbb Z$}}

\begin{document}
\title{Irreducibility of generalized Hermite-Laguerre Polynomials} %\\ \today}
\author[S. Laishram]{Shanta Laishram}
\address{Stat-Math Unit, India Statistical Institute\\
7, S. J. S. Sansanwal Marg, New Delhi, 110016, India}
\email{shanta@isid.ac.in}
\author[T. N. Shorey]{T. N. Shorey}
\address{Department of Mathematics\\
Indian Institute of Technology Bombay, Powai, Mumbai 400076, India}
\email{shorey@math.iitb.ac.in}
\maketitle
\pagenumbering{arabic}
\pagestyle{headings}

\section{Introduction}

Let $n$ and $1\le \al<d$ be positive integers with gcd$(\al, d)=1$.
Any positive rational $q$ is of the form $q=u+\frac{\al}{d}$ where $u$ is
a non-negative integer. For integers $a_0, a_1, \cdots a_n$, let
\begin{align*}
G(x):=G_q(x)=&a_nx^n+a_{n-1}(\al +(n-1+u)d)x^{n-1}+\cdots + \\
&a_1\left(\prod^{n-1}_{i=1}(\al +(i+u)d)\right)x+a_0
\left(\prod^{n-1}_{i=0}(\al +(i+u)d)\right).
\end{align*}
This is an extension of Hermite polynomials and generalized Laguerre polynomials. Therefore we call
$G(x)$ the generalized Hermite-Laguerre polynomial. For an integer $\nu >1$, we denote by
$P(\nu)$ the the greatest prime factor of $\nu$ and we put $P(1)=1$. We prove

\begin{theorem}\label{1/3}
Let $P(a_0a_n)\leq 3$ and suppose $2\nmid a_0a_n$ if degree of $G_{\frac{2}{3}}(x)$ is $43$.
Then $G_{\frac{1}{3}}$ and $G_{\frac{2}{3}}$ are irreducible except possibly when
$1+3(n-1)$ and $2+3(n-1)$ is a power of $2$, respectively where it can be  a product
of a linear factor times a polynomial of degree $n-1$.
\end{theorem}

\begin{theorem}\label{1/2}
Let $1\leq k<n$, $0\le u\le k$ and $a_0a_n\in \{\pm 2^t: t\geq 0, t\in \Z\}$.
Then $G_{u+\frac{1}{2}}$ does not have a factor of degree $k$
except possibly when $k\in \{1, n-1\}, u\ge 1$.
\end{theorem}

Schur \cite{schur} proved that  $G_{\frac{1}{2}}(x^2)$ with $a_n=\pm 1$ and $a_0=\pm 1$ are
irreducible and this implies the irreducibility of $H_{2n}$ where $H_m$ is the $m-$th Hermite
polynomial. Schur \cite{schur1} also established that Hermite polynomials $H_{2n+1}$ are $x$ times
an irreducible polynomial by showing that $G_{\frac{3}{2}}(x^2)$ with $a_n=\pm 1$ and $a_0=\pm 1$
is irreducible expect for some explicitly given finitely many values of $n$ where it can have a
quadratic factor.  Further Allen and Filaseta \cite{AlFil} showed that $G_{\frac{1}{2}}(x^2)$
with $a_1=\pm 1$ and $0<|a_n|<2n-1$ is irreducible. Finch and Saradha \cite{fsir}
showed that $G_{u+\frac{1}{2}}$ with $0\le u\le 13$ have no factor of degree $k\in [2, n-2]$
except for an explicitly given finite set of values of $u$ where it may have a factor of degree $2$.
%It is clear from the proof of Theorem \ref{1/3} that $G_{\frac{1}{3}}$ has no linear factor
%unless $1+3(n-1)$ is a power of $2$ and $G_{\frac{2}{3}}$ has no linear factor unless $2+3(n-1)$ is
%a power of $2$.

From now onwards, we always assume $d\in \{2, 3\}$. A new
ingredient in the proofs of Theorems \ref{1/3} and \ref{1/2} is the following result which
we shall prove in Section 3.

\begin{theorem}\label{dD}
Let $k\ge 2$ and $d=2, 3$. Let $m$ be a positive integer such that $d\nmid m$ and
$m>dk$. Then
\begin{align}\label{d>3k}
P(m(m+d)\cdots (m+d(k-1)))>\begin{cases}
3.5k & {\rm if} \ d=2 \ {\rm and} \ m\le 2.5k\\
4k & {\rm if} \ d=2 \ {\rm and} \ m>2.5k\\
3k & {\rm if} \ d=3
\end{cases}
\end{align}
unless $(m, k)\in \{(5, 2), (7, 2), (25, 2), (243, 2), (9, 4), (13, 5), (17, 6),  (15, 7), (21, 8), (19, 9)\}$
when $d=2$ and $(m, k)=(125, 2)$ when $d=3$.
\end{theorem}

If $d=2, 3$ and $m>dk$, this is an improvement of \cite{lsInd}.

In Section $4$, we shall combine Theorem \ref{dD} with the irreducibility
criterion from \cite{stirred}(see Lemma \ref{irmain}) to derive
Theorems \ref{1/3} and \ref{1/2}. This criterion come from Newton polygons.
If p is a prime and m is a nonzero integer, we define $\nu(m) = \nu_p(m)$
to be the nonnegative integer such that $p^{\nu(m)}|m$ and $p^{\nu(m)+1}\nmid m$.
We define $\nu(0) = +\infty$. Consider $f(x) =\sum^n_{j=0} a_jx^j\in \Z[x]$
with $a_0a_n\neq 0$ and let $p$ be a prime. Let $S$ be the following set of
points in the extended plane:
\begin{align*}
S = \{(0, \nu(a_n)), (1, \nu(a_{n-1})), (2, \nu(a_{n-2})), \cdots,
(n−1, \nu(a_1)),  (n, \nu(a_0))\}
\end{align*}
Consider the lower edges along the convex hull of these points. The
left-most endpoint is $(0, \nu(a_n))$ and the right-most endpoint is
$(n, \nu(a_0))$. The endpoints of each edge belong to S, and the slopes of
the edges increase from left to right. When referring to the "edges" of a
Newton polygon, we shall not allow two different edges to have the same
slope. The polygonal path formed by these edges is called the Newton
polygon of $f(x)$ with respect to the prime p. For the proof of
Theorems \ref{1/3} and \ref{1/2}, we use \cite[Lemma 10.1]{stirred} whose
proof depends on Newton polygons.

A part of this work was done when the authors were visiting Max-Planck Institute for
Mathematics in Bonn during August-October, 2008 and February-April, 2009, respectively.
We would like to thank the MPIM for the invitation and the hospitality. We also thank Pieter
Moree for his comments on a draft of this paper.
The authors are indebted to the referee for his suggestions and remarks
which improved the exposition of the paper.

\section{Preliminaries for Theorem \ref{dD}}

Let $m$ and $k$ be positive integers with $m>kd$ and gcd$(m,d) =1$. We write
\begin{align*}
\Delta (m, d, k)=m(m+d) \cdots (m+(k-1)d).
\end{align*}
For positive integers $\nu, \mu$ and $1\le l<\mu$ with gcd$(l, \mu)=1$, we write
\begin{align*}
\pi(\nu, \mu, l)=&\sum_{\underset{p\equiv l({\rm mod} \ \mu)}{p\le \nu}} 1, \
\pi(\nu)=\pi(\nu, 1, 1)\\
\theta(\nu, \mu, l)=&\sum_{\underset{p\equiv l({\rm mod} \ \mu)}{p\le \nu}}
\log p.
\end{align*}
Let $p_{i, \mu, l}$ denote the $i$th prime congruent to $l$ modulo $\mu$. Let
$\del_{\mu}(i, l)=p_{i+1, \mu, l}-p_{i, \mu, l}$ and
$W_\mu (i, l)=(p_{i, \mu, l}, p_{i+1, \mu, l})$.
Let $M_0=1.92367\times 10^{10}$.

We recall some well-known estimates on prime number theory.

\begin{lemma}\label{pix}
We have
\begin{enumerate}
\item[$(i)$]$\displaystyle{\pi (\nu) \leq \frac{\nu}{\log \nu}
\left(1+\frac{1.2762}{\log \nu}\right) \ {\rm for} \ \nu>1}$
\item[$(ii)$]$\displaystyle{\nu(1-\frac{3.965}{\log^2 \nu})\leq
\theta (\nu)<1.00008\nu \ {\rm for} \ \nu >1}$
\item[$(iii)$] $\sqrt{2\pi k}~e^{-k}k^{k}e^{\frac{1}{12k+1}} <k!<
\sqrt{2\pi k}~e^{-k}k^{k} e^{\frac{1}{12k}} \ {\rm for} \ k>1$
\item[$(iv)$] ${\rm ord}_p(k!)\geq \frac{k-p}{p-1}-
\frac{\log (k-1)}{\log p} \ {\rm for} \ k>1 \ {\rm and} \ p<k.$
\end{enumerate}
\end{lemma}

The estimates $(i), (ii)$ are due to Dusart \cite[p.14]{Thes}, \cite{Dus1}.
The estimate $(iii)$ is \cite[Theorem 6]{rob}. For a proof of $(iv)$, see
\cite[Lemma 2(i)]{shanta2}. \qed

The following lemma is due to Ramar\'e and Rumely \cite[Theorems 1, 2]{Rama}.

\begin{lemma}\label{ramar} Let $l\in \{1, 2\}$ . For $\nu_0\le 10^{10}$, we have
\begin{align}\label{lthe}
\theta(\nu, 3, l)\ge \begin{cases}
\frac{\nu}{2}(1-0.002238) \ {\rm for} \ \nu\ge 10^{10}\\
\frac{\nu}{2}\left(1-\frac{2\times 1.798158}
{\sqrt{\nu_0}}\right) \ {\rm for} \ 10^{10}>\nu\ge \nu_0
\end{cases}
\end{align}
and
\begin{align}\label{uthe}
\theta(\nu, 3, l)\le \begin{cases}
\frac{\nu}{2}(1+0.002238) \ {\rm for} \ \nu\ge 10^{10}\\
\frac{\nu}{2}\left(1+\frac{2\times 1.798158}{\sqrt{\nu_0}}\right) \
{\rm for} \ 10^{10}>\nu\ge \nu_0\end{cases}.
\end{align}
\end{lemma}

We derive from Lemmas \ref{pix} and \ref{ramar} the following result.

\begin{coro}\label{n<<}
Let $M_0<m\le 131\times 2k$ if $d=2$ and $6450\le m\le 10.6\times 3k$
if $d=3$. Then $P(\Delta(m, d, k))\ge m$.
\end{coro}

\begin{proof}
Let $M_0<m\le 131\times 2k$ if $d=2$ and $6450\le m\le 10.6\times 3k$
if $d=3$. Then $k\ge k_1$ where $k_1=7.34\times 10^7, 203$ when $d=2, 3$,
respectively. Let $1\le l<d$ and assume $m\equiv l($mod $d)$. We observe
that $P(\Delta(m, d, k)\ge m$ holds if
\begin{align*}%\label{thetam+k-m}
\theta(m+d(k-1), d, l)-\theta(m-1, d, l)=\sum_{\underset{p\equiv l(d)}
{m\leq p\le m+(k-1)d}}\log p >0.
%\\&\leq (\pi(m+(k-1)d, d, l)-\pi(m-1, d, l))\log (m+(k-1)d),
\end{align*}
%we see that $P(\Delta(m, d, k)\ge m$ is valid if
%$\theta(m+(k-1)d, d, l)>\theta(m-1, d, l)$.
Now from Lemmas \ref{pix} and \ref{ramar}, we have
\begin{align*}
\frac{\theta(m-1, d, l)}{\frac{m-1}{\phi(d)}}<\theta_1:=\begin{cases}
1.00008 & {\rm if} \ d=2\\
1+\frac{2\times 1.798158}{\sqrt{6450}} & {\rm if} \ d=3 %\le 1.04478
\end{cases}
\end{align*}
and
\begin{align*}
\frac{\theta(m+(k-1)d, d, l)}{\frac{m+(k-1)d}{\phi(d)}}>\theta_2:=
\begin{cases}
1-\frac{3.965}{\log^2 (10^{10})} & {\rm if} \ d=2\\
1-\frac{2\times 1.798158}{\sqrt{6450}}& {\rm if} \ d=3. %\ge .9522
\end{cases}
\end{align*}
\iffalse
\begin{align*}
\frac{\theta(m-1, 3, l)}{\frac{m-1}{2}}<
1+\frac{2\times 1.798158}{\sqrt{6450}}\le 1.04478
\end{align*}
and
\begin{align*}
\frac{\theta(m+3(k-1), 3, l)}{\frac{m+3(k-1)}{2}}>
1-\frac{2\times 1.798158}{\sqrt{6450}}\ge .9522
\end{align*}
\fi
Thus $P(\Delta(m, d, k)\ge m$ holds if
\begin{align*}
\theta_2(m+d(k-1))>\theta_1m
\end{align*}
i.e., if
\begin{align*}
\frac{d(k-1)}{m}%=\frac{dk}{m}(1-\frac{1}{k})
>\frac{\theta_1}{\theta_2}-1.%\ge .09375.
\end{align*}
This is true since for $k\ge k_1$, we have
\begin{align*}
\frac{dk(1-\frac{1}{k})}{\frac{\theta_1}{\theta_2}-1}\ge
\frac{dk(1-\frac{1}{k_1})}{\frac{\theta_1}{\theta_2}-1}>
(dk)\begin{cases}131.3 & {\rm if} \ d=2\\10.6 & {\rm if} \ d=3\end{cases}
\end{align*}
and $m$ is less than the last expression. Hence the assertion.
\end{proof}

Now we give some results for $d=2$. The next result follows from
Lemma \ref{pix} $(ii)$.

\begin{coro}\label{n<2k}
Let $d=2, k>1$ and $2k<m<4k$. Then
\begin{align}\label{4&3.5}
P(\Delta(m, d, k))>\begin{cases}
3.5k \ &{\rm if} \ m\le 2.5k\\
4k \ &{\rm if} \ m>2.5k
\end{cases}
\end{align}
unless $(m, k)\in \{(5, 2), (7, 2), (9, 4), (13, 5), (17, 6),  (15, 7),
(21, 8), (19, 9)\}$.
%Also, the set $\{m, m+2, \ldots, m+2(k-1)\}$ always contain a prime.
\end{coro}

\begin{proof}
We observe that the set $\{m, m+2, \ldots, m+2(k-1)\}$ contains all
primes between $3.5k$ and $4k$ if $m\le 2.5k$ and all primes
between $4k$ and $4.5k$ if $2.5k<m<4k$. Therefore \eqref{4&3.5} holds
if
\begin{align*}%\label{t4.54}
\begin{split}
\theta(4k)>\theta(3.5k) \ \ %&{\rm if} \ m>2.5k\\
{\rm and} \ \ \theta(4.5k)> \theta(4k).% \ \ &{\rm if} \ m\le 2.5k.
\end{split}\end{align*}
Let $(r, s)=(3.5, 4)$ or $(4, 4.5)$. Then from Lemma \ref{pix}, we
see that $\theta (sk)>\theta(rk)$ if
\begin{align*}
sk(1-\frac{3.965}{\log^2 (sk)})>1.00008\times rk
\end{align*}
or
\begin{align*}
\frac{s-1.00008r}{1.00008r}>\frac{s}{1.00008r}\frac{3.965}{\log^2 (sk)}
\end{align*}
or
\begin{align*}
k>\frac{1}{s}\exp\left(\sqrt{\frac{3.965s}{s-1.00008r}}\right).
\end{align*}
This is true for $k\ge 88$. Thus $k\le 87$. For $10\le k\le 87$, we check that
there is always a prime in the intervals $(3.5k, 4k)$ and $(4k, 4.5k)$ and
hence \eqref{4&3.5} follows in this case. For $2\le k\le 9$, the assertion
follows by computing $P(\Delta(m, 2, k))$ for each $2k<m<4k$.
\end{proof}

The following result concerns Grimm's Conjecture, \cite[Theorem 1]{grim}.

\begin{lemma}\label{grim}
Let $m\le M_0$ and $l$ be such that $m+1, m+2, \cdots , m+l$ are all
composite numbers. Then there are distinct primes $P_i$ such that
$P_i|(m+i)$ for each $1\le i\le l$.
\end{lemma}

As a consequence, we have

\begin{coro}\label{norpk}
Let $4k<m\le M_0$. Then either $P(\Delta(m, 2, k))>4k$ or
$P(\Delta(m, 2, k))\ge p_{k+1}$.
\end{coro}

\begin{proof}
If $m+2i$ is prime for some $i$ with $0\le i<k$, then the assertion holds
clearly since $P(\Delta(m, 2, k))\ge m+2i>4k$. Thus we suppose that
$m+2i$ is composite for all $0\le i<k$. Since $m$ is odd, we obtain
that $m+2i+1$ with $0\le i<k$ are all even and hence composite. Therefore
$m, m+1, m+2, \cdots , m+2k-1$ are all composite and hence,  by
Lemma \ref{grim}, there are distinct primes $P_j$ with $P_j|(m-1+j)$ for
each $1\le j\le 2k$. Therefore $\om(\Delta (m, 2, k))\ge k$ implying
$P(\Delta (m, 2, k))\ge p_{k+1}$.
\end{proof}

\begin{coro}\label{<10^10}
Let $d=2$ and $4k<m\le M_0$. Then $P(\Delta (m, 2, k))>4k$ for $k\ge 30$.
\end{coro}

\begin{proof}
By Corollary \ref{norpk}, we may assume that $P(\Delta (m, 2, k))\ge p_{k+1}$.
By Lemma \ref{pix}, we get $p_{k+1}\ge k\log k$ which is $>4k$ for $k\ge 60$. For
$30\le k<60$, we check that $p_{k+1}>4k$. Hence the assertion follows.
\end{proof}

The following result follows from \cite[Tables IIA, IIIA]{lehmer}.

\begin{lemma}\label{leh31}
Let $d=2$, $m>4k$ and $2\le k\le 37, k\neq 35$. Then
$P(\Delta (m, 2, k))>4k$.
\end{lemma}

\begin{proof}
The case $k=2$ is immediate from \cite[Table IIA]{lehmer}. Let $k\ge 3$
and $m\ge 4k$. For $m$ and $1\le i<k$ such that $m+2i=N$ with $N$ given in
\cite[Tables IIA, IIIA]{lehmer}, we check that $P(\Delta (m, 2, k))>4k$.
Hence assume that $m+2i$ with $1\le i<k$ is different from those $N$
given in \cite[Tables IIA, IIIA]{lehmer}.

For every prime $31<p\le 4k$, we delete a term in
$\{m, m+2, \cdots, m+2(k-1)\}$ divisible by $p$. %Let $L=\{m+2i: P(m+2i)\le 31\}$ and $l=|L|$.
Let $i_1<i_2<\ldots <i_l$ be such that $m+2i_j$ is in the remaining set where
$l\ge k-(\pi(4k)-\pi(31))$. From \cite[Tables IIA, IIIA]{lehmer}, we observe that $i_{j+1}-i_j\ge 3$
implying $k-1\ge i_l-i_1\ge 3(l-1)\ge 3(k-\pi(4k)+10)$. However we find that the inequality
$k-1\ge 3(k-\pi(4k)+10)$ is not valid except when $k=28, 29$. Hence the assertion of the Lemma is
valid except possibly for $k=28, 29$.

Therefore we may assume that $k=28, 29$. Further we suppose that $l=k-(\pi(4k)-\pi(31))=10$ otherwise
$3(l-1)\ge 30>k-1$, a contradiction. Thus we have either $i_{10}-i_1=27$ implying
$i_1=0, i_{j+1}=i_j+3=3j$ for $1\le j\le 9$ or $i_1=1, i_{j+1}=i_j+3=3j+1$ for $1\le j\le 9$
or $i_{10}-i_1=28$ implying $i_1=0, i_{j+1}=\begin{cases}3j &
{\rm if} \ 1\le j\le r\\3j+1 & {\rm if} \ r<j\le 9\end{cases}$ for some $r\ge 1$.
%For $k=28$, the only possibility is $i_1=0, i_{j+1}=i_j+3=3j$ for $1\le j\le 9$.
Let $X=m+2i_1-6$. Note that $X$ is odd since $m$ is odd. Also $X\ge 4k+1-6\ge 107$. We have either
\begin{align}\label{28X}
P((X+6)\cdots (X+54)(X+60))\le 31
\end{align}
or there is some $r\ge 1$ for which
\begin{align}\label{29X}
P((X+6)\cdots (X+6r)(X+6(r+1)+2)\cdots (X+60+2))\le 31.
\end{align}
Note that \eqref{28X} is the only possibility when $k=28$. Now we consider \eqref{28X}.
Suppose $3|X$. Then putting $Y=\frac{X}{3}$, we get $P((Y+2)\cdots (Y+18)(Y+20))\le 31$ which
implies $Y+2<20$ by Corollary \ref{n<2k} and Lemma \ref{leh31} with $k=10$. Since
$X+6\ge m\ge 113$, we get a contradiction. Hence we may assume that $3\nmid X$. Then
$3\nmid (X+6)\cdots (X+54)(X+60)$. After deleting terms $X+6i$ divisible by primes $11\le p\le 31$, we
are left with three terms divisible by primes $5$ and $7$ and hence $m\le X+6\le 35$ which is again
a contradiction. Therefore \eqref{28X} is not possible.

Now we consider \eqref{29X} which is possible only when $k=29$. Since $X+6=m>4k=116$, we have
$X>110$. Suppose $r=1, 9$. Then we have $P((X+12+2)\cdots (X+54+2)(X+60+2))\le 31$ if $r=1$ and
$P((X+6)\cdots (X+54))\le 31$ if $r=9$. Putting $Y=X+8$ in the first case and $Y=X$ in the
latter, we get $P((Y+6)\cdots (Y+54))\le 31$. Suppose $3|Y$. Then putting $Z=\frac{Y}{3}$, we get
$P((Z+2)\cdots (Z+18))\le 31$ which implies $Z+2\le 18$ by Corollary \ref{n<2k} and Lemma \ref{leh31}
with $k=9$. Since $Z+2\ge \frac{X}{3}>\frac{110}{3}$, we get a contradiction. Hence we may assume that
$3\nmid Y$. Then $3\nmid (Y+6)\cdots (Y+54)$. After deleting terms $Y+6i$ divisible by primes
$11\le p\le 31$, we are left with two terms divisible by primes $5$ and $7$ only. Let $Y+6i=5^{a_1}7^{b_1}$ and
$Y+6j=5^{a_2}7^{b_2}$ where $b_1\le 1<b_2$ and $a_2\le 1<a_1$. Since $|i-j|\le 8$, the equality
$6(i-j)=5^{a_1}7^{b_1}-5^{a_2}7^{b_2}$ implies $5^a-7^b=\pm 6, \pm 12, \pm 18, \pm 24, \pm 36, \pm 48$.
By taking modulo $6$, we get $(-1)^a\equiv 1$ modulo $6$ implying $a$ is even. Taking modulo $8$ again,
we get either
\begin{align*}
b \ {\rm is \ even}, \ 5^a-7^b=(5^{\frac{a}{2}}-7^{\frac{b}{2}})(5^{\frac{a}{2}}+7^{\frac{b}{2}})=\pm 24, \pm 48
\end{align*}
giving
\begin{align}\label{25-49}
5^a=25, 7^b=49
\end{align}
or
\begin{align*}
b \ {\rm is \ odd}, \ 5^a-7^b=-6, 18.
\end{align*}
Let $5^a-7^b=-6$. Considering modulo $5$, we get $2^b\equiv 1$ implying $4|b$, a contradiction. Let $5^a-7^b=18$. By
considering modulo $7$ and modulo $9$ and since $a$ is even, we get $3|(a-2)$ and $3|(b-1)$ implying
$(5^{\frac{a+1}{3}})^3+35(-7^{\frac{b-1}{3}})^3=90$. Solving the Thue equation $x^3+35y^3=90$ gives
$x=5, y=-1$ or $25-7=18$ is the only solution. Hence $6\cdot 3=25-7=X+6i-(X+6j)$.
Also the solution \eqref{25-49} implies $-6\cdot 4=25-49=X+6i-(X+6j)$. Thus $X\le 25$ which is not possible.

Assume now that $2\le r\le 8$. Then $P((X+6)(X+12)(X+56)(X+62))\le 31$. Suppose $3|X(X+2)$.
Putting $Y=\frac{X+6}{3}$ if $3|X$ and $Y=\frac{X+56}{3}$ if $3|(X+2)$, we get either
$P(Y(Y+2)(3Y+50)(6Y+56))\le 31$ or $P(Y(Y+2)(3Y-50)(3Y-44))\le 31$. In particular
$P(Y(Y+2))\le 31$. For $Y=N-2$ given by \cite[Table IIA]{lehmer} such that $P(Y(Y+2))\le 31$
, we check that $P((3Y+50)(3Y+56))>31$ and $P((3Y-50)(3Y-44))>31$ except when $Y\in \{55, 145, 297, 1573\}$.
This gives $m=X+6=3Y-50$ and then we further check that $P(\Delta (m, 2, k))>116$. %This is a contradiction.
Hence we suppose $3\nmid X(X+2)$. Then $3\nmid (X+6)\cdots (X+6r)(X+6(r+1)+2)\cdots (X+60+2)$. If a prime power
$p^a$ divides two terms of the product, then $p^a|(X+6j), p^a|(X+6i)$ or $p^a|(X+6j+2), p^a|(X+6i+2)$ or
$p^a|(X+6j), p^a|(X+6i+2)$ for some $i, j$. Hence $p^a|6(i-j)$ or $p^a|6(i-j)+2$. Since
$1\le j<i\le 10$, we get $p^a\in \{5, 7, 11, 13, 19, 25\}$. After deleting terms divisible by
primes $5\le p\le 31$ to their highest powers, we are left with two terms such that their product divides
$25\cdot 7\cdot 11\cdot 13\cdot 19$ and hence $X+6\le \sqrt{25\cdot 7\cdot 11\cdot 13\cdot 19}$
or $X+6\le 689$. We check that $P((X+6)(X+12)(X+56)(X+62))>31$ for $110\le X\le 683$ except when
$X\in \{113, 379\}$. Further we check that $P(\Delta (m, 2, k))>116$ for $m=X+6$. Hence the result.
\end{proof}

The remaining results in this section deal with the case $d=3$. The first one
is a computational result.

\begin{lemma}\label{diff}
Let $l\in \{1, 2\}$.  If $p_{i, 3, l}\le 6450$, then $\del_3(i, l)\le 60$.
\end{lemma}

As a consequence, we obtain

\begin{coro}\label{<20dD3}
Let $d=3$ and $3k<m\leq 6450$ with gcd$(m, 3)=1$. Then \eqref{d>3k} holds
unless $(m, k)=(125, 2)$.
\end{coro}

\begin{proof}
For $k\leq 20$, it follows by direct computation. For $k>20$, \eqref{d>3k} follows as
$3(k-1)\geq 60$ and, by Lemma \ref{diff}, the set $\{m+3i: 0\leq i<k\}$ contains a
prime.
\end{proof}

We shall also need the following result of Nagell \cite{nag}(see
\cite{cao}) on diophantine equations.

\begin{lemma}\label{nagel}
Let $a, b, c\in \{2, 3, 5\}$ and $a<b$. Then the solutions of
\begin{align*}
a^x+b^y=c^z \ {\rm in \ integers} \ x>0, y>0, z>0
\end{align*}
are given by
\begin{align*}
(a^x, b^y, c^z)\in \{&(2, 3, 5), (2^4, 3^2, 5^2), (2, 5^2, 3^3),\\
&(2^2, 5, 3^2), (3, 5, 2^3), (3^3, 5, 2^5), (3, 5^3, 2^7)\}.
\end{align*}
\end{lemma}

As a corollary, we have
\begin{coro}\label{upto7}
Let $X>80, 3\nmid X$ and $1\le i\le 7$. Then the solutions of
\begin{align*}
P(X(X+3i))=5 \ \ {\rm and} \ 2|X(X+3i)
\end{align*}
are given by
\begin{align*}
(i, X)\in \{(1, 125), (2, 250), (4, 500), (5, 625)\}.
\end{align*}
\end{coro}

\begin{proof}
Let $1\le i\le 7$. We observe that $2|X, 2|(X+3i)$ only if $X$ and $i$
are both even and $5|X, 5|(X+3i)$ only if $i=5$. Let the positive
integers $r, s$ and $\del=$ord$_2(i)\in \{0, 1, 2\}$ be given by
\begin{align}\label{not5}
X=2^{r+\del}, \ X+3i=2^{\del} 5^s  \ \ {\rm or} \ \
X=2^{\del} 5^s, \  X+3i=2^{r+\del} \ {\rm if} \ i\neq 5
\end{align}
and
\begin{align}\label{i=5}
X=5^{s+1}, \  X+3i=5\times 2^r\ \ {\rm or} \ \
X=5\times 2^r, \  X+3i=5^{s+1} \ {\rm if} \ i=5,
\end{align}
where $r+2\ge r+\del \ge 7$ and $s\ge 2$ since $X>80$. Hence we have
\begin{align}\label{not5eqn}
2^r-5^s=\pm\left(\frac{X+3i}{2^{{\rm ord}_2(i)}\cdot 5^{{\rm ord}_5(i)}}-
\frac{X}{2^{{\rm ord}_2(i)}\cdot 5^{{\rm ord}_5(i)}}\right)=
\pm 3\times \frac{i}{2^{{\rm ord}_5(i)}\cdot 5^{{\rm ord}_5(i)}}.
\end{align}

Let $i\in \{1, 2, 4, 5\}$. Then $2^r -5^s =\pm 3$. By Lemma
\ref{nagel}, we have $2^r=2^7, 5^s=5^3$ and $2^7-5^3=3$ implying
$X=2^{{\rm ord}_2(i)}\cdot 5^{3+{\rm ord}_5(i)}$ and
$X+3i=2^{7+\del}\cdot 5^{{\rm ord}_5(i)}$. These give the solutions stated
in the Corollary.

Let $i\in \{3, 6\}$. Then $2^r -5^s =\pm 9=\pm 3^2$. Since
min$(2^r, 5^s)>16$, we observe from Lemma \ref{nagel} that there is no
solution.

Let $i=7$. Then $2^r -5^s =\pm 21$. Let $s$ be even.
Since $2^r>16$, taking modulo $8$, we find that $-1\equiv \pm 21($
modulo $8)$ which is not possible. Hence $s$ is odd. Then
$2^r -5^s \equiv 2^r + 2^s \equiv 0$ modulo $7$. Since
$2^r, 2^s \equiv 1, 2, 4$ modulo $7$, we get a contradiction.
\end{proof}

\section{Proof of Theorem \ref{dD}}\label{dDProof}

Let $D=4, 3$ according as $d=2, 3$, respectively. Let $v=\frac{m}{dk}$.
Assume that
\begin{align}\label{Dk+u}
P(\Delta (m, d, k))=P(m(m+d)\cdots (m+(k-1)d)<Dk.
\end{align}
Then
\begin{align}\label{piDk}
\om(\Delta (m, d, k))\le \pi(Dk)-1.
\end{align}
For every prime $p\le Dk$ dividing $\Delta $, we delete a term
$m+i_pd$ such that ord$_p(m+i_pd)$ is maximal. Note that $p|(m+id)$
for at most one $i$ if $p\ge k$. Then we are left with a set $T$ with
$1+t:=|T|\ge k-\pi(Dk)+1:=1+t_0$. Let $t_0\geq 0$ which we assume in this
section to ensure that $T$ is non-empty. We arrange the
elements of $T$ as $m+i'_0d<m+i'_1d<\cdots <m+i'_{t_0}d<..<m+i'_td$. Let
\begin{align}\label{P0}
{\frak P}:=\displaystyle{\prod^{t_0}_{\nu =0}} (m+i'_{\nu}d) \geq
d^{k-\pi(Dk)+1}\prod^{k-\pi(Dk)}_{i=0}(vk+i).
\end{align}
We now apply \cite[Lemma 2.1, (14)]{shanta2} to get
\begin{align*}
{\frak P}\leq (k-1)! d^{-{\rm ord}_{d}(k-1)!}.
\end{align*}
Comparing the upper and lower bounds of ${\frak P}$, we have
\begin{align*}%\label{E01}
d^{\pi(Dk)}\geq \frac{d^{k+1}
\prod^{k-\pi(Dk)}_{i=0}(vk+i)}{(k-1)! d^{-{\rm ord}_{d}(k-1)!}}
%\ge \frac{d^{k+1}d^{{\rm ord}_{d}(k-1)!}(vk)^{k+1-\pi(Dk)}}{(k-1)!}.
\end{align*}
which imply
\begin{align}\label{E0}
d^{\pi(Dk)}\geq \frac{d^{k+1}d^{{\rm ord}_{d}(k-1)!}(vk)^{k+1-\pi(Dk)}}{(k-1)!}.
\end{align}
\iffalse
This can be rewritten as
\begin{align}\label{E1}
(vdk)^{\pi(Dk)}>\frac{(vdk)^{k+1}d^{{\rm ord}_{d}(k-1)!}}{(k-1)!}.
\end{align}
\fi
By using the estimates for ord$_d((k-1)!)$ and $(k-1)!$ given in Lemma \ref{pix}, we obtain
\begin{align*}\begin{split}
(vdk)^{\pi(Dk)}&>\frac{(vdk)^{k+1}d^{(k-d)/(d-1)}(k-1)^{-1}}
{\sqrt{2(k-1)\pi}(\frac{k-1}{e})^{k-1}exp(\frac{1}{12(k-1)})}\\
&=\left(evd^{\frac{d}{d-1}}\frac{k}{k-1}\right)^k
\frac{v\sqrt{k}}{ed^{1/(d-1)}\sqrt{2\pi}}
\sqrt{\frac{k}{k-1}}exp(-\frac{1}{12(k-1)})
\end{split}
\end{align*}
implying
\begin{align}\label{E3}
\pi(Dk)>
\frac{k\log (evd^{\frac{d}{d-1}})+(k+\frac{1}{2})\log (\frac{k}{k-1})-
\frac{1}{12(k-1)}+\frac{1}{2}\log \frac{v^2k}{2\pi e^2d^{\frac{2}{d-1}}}}
{\log (vdk)}.
\end{align}
Again by using the estimates for $\pi(\nu)$ given in Lemma \ref{pix} and
$\frac{\log (vdk)}{\log (Dk)}=1+\frac{\log \frac{vd}{D}}{\log (Dk)}$, we derive
\begin{align}\begin{split}\label{E4}
0>&\frac{1}{2}\log \frac{v^2k}{2\pi e^2d^{\frac{2}{d-1}}}-
\frac{1}{12(k-1)}\\
&+k\left(\log (evd^{\frac{d}{d-1}})-D\left(1+\frac{\log \frac{vd}{D}}{\log (Dk)}\right)
\left(1+\frac{1.2762}{\log (Dk)}\right)\right).
\end{split}\end{align}
\iffalse
\begin{align}\label{E40}
0>(k+\frac{1}{2})\log (\frac{k}{k-1})+F(v, k)
\end{align}
where
\begin{align*}\begin{split}
F(k, v):=&\frac{1}{2}\log \frac{v^2k}{2\pi e^2d^{\frac{2}{d-1}}}-
\frac{1}{12(k-1)}\\
&+k\left(\log (evd^{\frac{d}{d-1}})-D\left(1+\frac{\log \frac{vd}{D}}{\log (Dk)}\right)
\left(1+\frac{1.2762}{\log (Dk)}\right)\right).
\end{split}
\end{align*}
\fi

Let $v$ be fixed with $vd\geq D$. Then expression
\begin{align*}
F(k, v):=\log (evd^{\frac{d}{d-1}})-D\left(1+\frac{\log \frac{vd}{D}}{\log (Dk)}\right)
\left(1+\frac{1.2762}{\log (Dk)}\right)
\end{align*}
is an increasing function of $k$. Let $k_1:=k_1(v)$ be such that $F(k, v)>0$ for all $k\ge k_1$.
Then we observe that the right hand side of \eqref{E4} is an increasing function for
$k\ge k_1$. Let $k_0:=k_0(v)\ge k_1$ be such that the right hand side of \eqref{E4} is positive.
Then \eqref{E4} is not valid for all $k\ge k_0$ implying \eqref{E3} and hence \eqref{E0} are not
valid for all $k\ge k_0$.

Also for a fixed $k$, if \eqref{E4} is not valid at some $v=v_0$, then \eqref{E0} is also not valid at
$v=v_0$. Observe that for a fixed $k$, if \eqref{E0} is not valid at some $v=v_0$, then \eqref{E0} is also
not valid when $v\geq v_0$.

Therefore for a given $v=v_0$ with $v_0d\geq D$, the inequality \eqref{E0} is not valid for all $k\geq k_0(v_0)$
and $v\geq v_0$.
%Therefore whenever \eqref{E4} is not valid for $k=k_0$ at some $v=v_0$ with $v_0d\geq D$, the inequality
%\eqref{E0} is not valid for all $k\geq k_0$ and $v\geq v_0$.

\section*{3(a). Proof of Theorem \ref{dD} for the case $d=3$}\label{k<9}

Let $d=3$ and let the assumptions of Theorem \ref{dD} be satisfied. Let
$2\le k\le 11$ and $m>3k$. Observe that $k-\pi(3k)+1=0$ for
$k\le 8$ and $k-\pi(3k)+1=1$ for $9\le k\le 11$. If $T\neq \phi$, then
$m\le 2^3\times 5\times 7=280$.

By Corollary \ref{<20dD3}, we may assume that $2\le k\le 8$,
$m\geq 6450$ and $T=\phi$. Further $i_p$ exists for each prime $p\le 3k$,
$p\neq 3$ and $i_p\neq i_q$
for $p\neq q$ otherwise $|T|\ge k-\pi(3k)+1+1>0$. Also $pq\nmid (m+id)$
for any $i$ whenever $p, q\ge k$ otherwise $T\neq \phi$. Thus
$P((m+3i_2)(m+3i_5))=5$ if $k<8$. For $k=8$, we get
$P((m+3i_2)(m+3i_5))\le 7$ with $P((m+3i_2)(m+3i_5))=7$ only if $7|m$ and
$\{i_2, i_5\}\cap \{0, 7\}\neq \phi$.

Let $k\le 7$ or $k=8$ with $P((m+3i_2)(m+3i_5))=5$. Let
$j_0=$min$(i_2, i_5)$, $X=m+3j_0$ and $i=|i_2-i_5|$. Then $X\geq 6450$ and
this is excluded by Corollary \ref{upto7}.

Let $k=8$ and $P((m+3i_2)(m+3i_5))=7$. Then $7|m$ and
$\{i_2, i_5\}\cap \{0, 7\}\neq \phi$. Hence $i_7=0$ or $7$ and
$7\in \{i_2, i_5\}$ if $i_7=0$ and $0\in \{i_2, i_5\}$ if $i_7=7$. If
$5\nmid m(m+21)$, then $\{i_2, i_7\}=\{0, 7\}$ and either
\begin{align*}
m=7\times 2^r, \ m+21=7^{1+s} \ \ {\rm or} \ \
m=7^{1+s}, \ m+21=7\times 2^r
\end{align*}
implying $2^r -7^s =\pm 3$. Since $2^r \ge \frac{m}{7}>40$, we get
by taking modulo $8$ that $(-1)^{s+1}\equiv \pm 3$ which is a
contradiction. Thus $5|m(m+21)$ implying $2\times 5\times 7|m(m+21)$. By
taking the prime factorization, we obtain
\begin{align*}
m=2^{a_0}5^{b_0}7^{c_0}, \ m+21=2^{a_1}5^{b_1}7^{c_1}
\end{align*}
with min$(a_0, a_1)=$min$(b_0, b_1)=0$, min$(c_0, c_1)=1$ and further
$b_0+b_1=1$ if $i_2\in \{0, 7\}$ and $a_0+a_1\le 2$ if $i_5\in \{0, 7\}$.
From the identity $\frac{m+21}{7}-\frac{m}{7}=3$, we obtain one of
\begin{align*}
&(i) \ 2^a-5\cdot 7^c=\pm 3 \ \ {\rm or} \ \ (ii) \ 5\cdot 2^a-7^c=\pm 3 \\
{\rm or} \ &(iii) \ 5^b-2^\del \cdot 7^c=\pm 3 \ \ {\rm or} \ \
(iv) \ 2^\del \cdot 5^b-7^c=\pm 3
\end{align*}
with $\del \in \{1, 2\}$. Further from $m\geq 6450$, we obtain $c \ge 3$ and
\begin{align}\label{algab}
a\ge 9, a\ge 7, b \ge 4, b \ge 3
\end{align}
according as $(i), (ii), (iii), (iv)$ hold, respectively. These equations give
rise to a Thue equation
\begin{align}\label{thue}
X^3+AY^3=B
\end{align}
with integers $X, Y, A>0, B>0$ given by
\begin{center}
\begin{tabular}{|c|c|c|c|c|c|c|}\hline
& $\underset{({\rm mod} \ 3)}{c}$ & Equation & $A$ & $B$ & $X$ & $Y$\\
\hline
$(i)$ & $0, 1$ & $2^{a}-5\cdot 7^{c}=\pm 3$ &
$5\cdot 2^{a'}\cdot 7^{c'}$ & $3\cdot 2^{a'}$ &
$\pm 2^{\frac{a+a'}{3}}$ & $\pm 7^{\frac{c-c'}{3}}$
\\ \hline
$(ii)$ & $0, 1$ & $5\cdot 2^{a}-7^{c}=\pm 3$ &
$25\cdot 2^{a'}\cdot 7^{c'}$ & $75\cdot 2^{a'}$ &
$\pm 5\cdot 2^{\frac{a+a'}{3}}$ & $\pm 7^{\frac{c-c'}{3}}$ \\ \hline
$(iii)$ & $0, 1$ & $5^{b}-2^\del \cdot 7^{c}=\pm 3$ &
$2^\del \cdot 5^{b'}\cdot 7^{c'}$ & $3\cdot 5^{b'}$ &
$\pm 5^{\frac{b+b'}{3}}$ & $\pm 7^{\frac{c-c'}{3}}$
\\ \hline
$(iv)$ & $0, 1$ & $2^\del \cdot 5^{b}-7^{c}=\pm 3$ &
$2^{3-\del}\cdot 5^{b'}\cdot 7^{c'}$ & $2^{3-\del}\cdot 5^{b'}\cdot 3$ &
$\pm 2\cdot 5^{\frac{b+b'}{3}}$ & $\pm 7^{\frac{c-c'}{3}}$ \\ \hline
$(v)$ & $2$ & $2^{a}-5\cdot 7^{c}=\pm 3$ & $175\cdot 2^{a'}$ &
$525$ & $\pm 5\cdot 7^{\frac{c+1}{3}}$ & $\pm 2^{\frac{a-a'}{3}}$
\\ \hline
$(vi)$ & $2$ & $5\cdot 2^{a}-7^{c}=\pm 3$ & $35\cdot 2^{a'}$ &
$21$ & $\pm 7^{\frac{c+1}{3}}$ & $\pm 2^{\frac{a-a'}{3}}$ \\ \hline
$(vii)$ & $2$ & $5^{b}-2^\del \cdot 7^{c}=\pm 3$ &
$2^{3-\del} \cdot 5^{b'}\cdot 7$ & $21\cdot 2^{3-\del}$ &
$\pm 2\cdot 7^{\frac{c+1}{3}}$ & $\pm 5^{\frac{b-b'}{3}}$ \\ \hline
$(viii)$ & $2$ & $2^\del \cdot 5^{b}-7^{c}=\pm 3$ &
$2^{\del}\cdot 5^{b'}\cdot 7$ & $21$ & $\pm 7^{\frac{c+1}{3}}$ &
$\pm 5^{\frac{b-b'}{3}}$ \\ \hline
\end{tabular}
\end{center}
where $0\le a', b'<3$ are such that $X, Y$ are integers and $c'=0, 1$
according as $c ($mod $3)=0, 1$, respectively. For example,
$2^a-5\cdot 7^c=\pm 3$ with $c\equiv 0, 1($mod $3)$ implies
$(\pm 2^{\frac{a+a'}{3}})^3+5\cdot 2^{a'}7^{c'}(\pm 7^{\frac{c-c'}{3}})^3=
3\cdot 2^{a'}$ where $a'$ is such that $3|(a+a')$. This give a Thue
equation \eqref{thue} with $A=5\cdot 2^{a'}7^{c'}$ and $B=3\cdot 2^{a'}$.

By using \eqref{algab}, we see that at least two of
 \begin{align}\label{ordXY}
{\rm ord}_2(XY)\ge 2 \ {\rm or} \ {\rm ord}_5(XY)\ge 1 \ {\rm or} \
{\rm ord}_7(XY)\ge 1
\end{align}
hold except for $(vi)$ and $(viii)$ where ord$_2(XY)\ge 1$, ord$_7(XY)\ge 1$
in case of $(vi)$ and ord$_2(XY)=0$, ord$_7(XY)\ge 1$ in case of
$(viii)$. Using the command
\begin{center}
T:=Thue$(X^3+A)$; Solutions$(T, B)$;
\end{center}
in \emph{Kash}, we compute all the solutions in integers $X, Y$ of
the above Thue equations. We find that none of solutions of Thue
equations satisfy \eqref{ordXY}.

Hence we have $k\ge 12$. For the proof of Theorem \ref{dD}, we may
suppose from Corollaries \ref{<20dD3} and \ref{n<<} that
\begin{align}\label{nbd3}
m\ge \max(6450, 10.6\times 3k).
\end{align}

Let $12\le k\le 19$. Since $t_0\ge 1, 2$ for $12\le k\le 16$ and
$17\le k\le 19$, respectively, we have
\begin{align*}
m&\le \sqrt{{\frak P}}\le \sqrt{4\times 8\times 5^2\times 7^2\times
11\times 13}<6450 \ \hspace{1.5cm}  {\rm if} \ 12\le k\le 16\\
m&\le \sqrt[3]{{\frak P}}\le \sqrt[3]{4\times 8\times 16\times 5^3\times
7^2\times 11\times 13\times 17}<6450 \ \ {\rm if} \ 17\le k\le 19.
\end{align*}
This is not possible by \eqref{nbd3}.

Thus $k\ge 20$. Then $m\ge 6450$ and $v\ge 10.6$ by \eqref{nbd3} satisfying
$v_0d\geq D=d=3$. Now we check that $k_0\leq 180$ for $v=10.6$. Therefore
\eqref{E0} is not valid for $k\geq 180$ and $v\geq 10.6$. Thus
$k<180$. Further we check that \eqref{E3} is not valid for $20\leq k<180$
at $v=\frac{6450}{3k}$ except when $k\in \{21, 25, 28, 37, 38\}$.
Hence \eqref{E0} is not valid for $20\leq k<180$ when $v\geq \frac{6450}{3k}$
except when $k\in \{21, 25, 28, 37, 38\}$. Thus it suffices to consider
$k\in \{21, 25, 28, 37\}$ where we check that \eqref{E0} is not valid
at $v=\frac{6450}{3k}$ and hence it is not valid for all
$v\ge \frac{6450}{3k}$. Finally we consider $k=38$ where we find
that \eqref{E0} is not valid at $v=\frac{8000}{3k}$. Thus
$m<8000$. For $l\in \{1, 2\}$ and $p_{i, 3, l}\leq 8000$, we find
that $\del_3(i, 3, l)<90$ implying the set $\{m, m+3, \ldots, m+3(38-1)\}$
contains a prime. Hence the assertion follows since $m>3k$.
\qed

\section*{3(b). Proof of Theorem \ref{dD} for $d=2$}\label{d24}

Let $d=2$ and let the assumptions of Theorem \ref{dD} be satisfied. The
assertion for Theorem \ref{dD} with $k\ge 2$ and $m\le 4k$ follows
from Corollary \ref{n<2k}. Thus $m>4k$. For $2\le k\le 37$, $k\neq 35$,
Lemma \ref{leh31} gives the result. Hence for the proof of
Theorem \ref{dD}, we may suppose that $k=35$ or $k\ge 38$. Further from
Corollaries \ref{n<<} and \ref{<10^10}, we may assume that
\begin{align}\label{nbd2}
m\ge \max(M_0, 131\times 2k).
\end{align}

Let $k=35, 38$. Then $t_0=1, 2$ for $k=35, 38$, respectively and we have
\begin{align*}
m&\le \sqrt{{\frak P}}\le \sqrt{27\cdot 9\cdot 25\cdot 5\cdot 7^2\cdot
11^2\cdot 13^2\cdot 17^2\cdot 19\cdot 23\cdot 29\cdot 31} \ \
\hspace{.7cm} \ <10^{10} \ {\rm if} \ k=35\\
m&\le \sqrt[3]{{\frak P}}\le \sqrt[3]{27\cdot 9^2\cdot 25\cdot 5^2\cdot 7^3\cdot
11^3\cdot 13^2\cdot 17^2\cdot 19\cdot 23\cdot 29\cdot 31\cdot 37}<10^{10} \ {\rm if} \ k=38.
\end{align*}
This is not possible by \eqref{nbd2}.

Thus we assume that $k\ge 39$. Let $v\ge 131$ and we check that $k_0\leq 500$ for $v=131$.
Therefore \eqref{E0} is not valid for $k\geq 500$ and $v\geq 131$. Hence
from \eqref{nbd2}, we get $k<500$. Further $v\ge \frac{M_0}{2\times 500}\ge 10^7$.
We check that $k_0\leq 70$ at $v=10^7$ implying \eqref{E0} is not valid for $k\geq 70$ and
$v\geq 10^7$. Thus $k<70$. For each $39\le k<70$, we find that \eqref{E0} is not valid at
$v=\frac{M_0}{2k}$ and hence for all $v \ge \frac{M_0}{2k}$. This is a contradiction.
\qed

\section{Proof of Theorems \ref{1/3} and \ref{1/2}}

Recall that $q=u+\frac{\al}{d}$ with $1\le \al<d$. We observe that if $G(x)$ has a
factor of degree $k$, then it has a cofactor of degree $n-k$. Hence we may assume from
now on that if $G(x)$ has a factor of degree $k$, then $k\leq \frac{n}{2}$. The following result
is \cite[Lemma 10.1]{stirred}.

\begin{lemma}\label{irmain}
Let $1\le k\le \frac{n}{2}$ and
\begin{align*}
d\le 2\al +2 \ \ {\rm if} \ \ (k, u)=(1, 0).
\end{align*}
If there is a prime $p$ with
\begin{align*}
p|(\al +(n+u-k)d)\cdots (\al +(n+u-1)d), \ \ p\nmid a_0a_n.
\end{align*}
such that %Assume that the prime $p$ satisfies
\begin{align*}
p\ge \begin{cases}
(k+u-1)d +\al +1 & \ {\rm if} \ u>0\\
(k+u-1)d +\al +2 & \ {\rm if} \ u=0
\end{cases}
\end{align*}
Then $G(x)$ has no factor of degree $k$.
\end{lemma}

Let $d=3$. By putting $m=\al +3(n-k)$ and taking $p=P(\Delta (m, 3, k))$, we find
from Lemma \ref{irmain} and Theorem \ref{dD} that $G_{\frac{1}{3}}$ and
$G_{\frac{2}{3}}$ does not have a factor of degree $k\ge 2$ except possibly
when $k=2, \al=2, m=2+3(n-2)=125$. This gives $n=43$ and we use \cite[Lemma 2.13]{stirred}
with $p=2, r=2$ to show that $G_{\frac{2}{3}}$ do not have a factor of degree $2$.
Further except possibly when $m=\al +3(n-1)=2^l$ for positive integers $l$,
$G_{\frac{1}{3}}$ and $G_{\frac{2}{3}}$ do not have a linear factor.
This proves Theorem \ref{1/3}.

Let $d=2$. Let $k=1, u=0$. We have $P(1+2(n-1))\geq 3$ and hence
taking $p=P(1+2(n-1))$ in Lemma \ref{irmain}, we find that $G_{\frac{1}{2}}$ does not
have a factor of degree $1$. Hence from now on, we may suppose that $k\ge 2$ and
$0\leq u\leq k$. For $(m, k)\in \{((5, 2), (7, 2), (9, 4), (13, 5), (17, 6), (15, 7),
(21, 8), (19, 9)\}$, we check that $P(\Delta (m, 2, k))\ge m$. For $0\le u\le k$, by putting
$m=1+2(n+u-k)$, we find from $n\ge 2k$ and Theorem \ref{dD} that
\begin{align*}
P(\Delta (m, 2, k))>2(k+u)=\begin{cases}
\min (2(k+u), 3.5k) & {\rm if} \ u\le 0.5k\\
\min (2(k+u), 4k) & {\rm if} \ 0.5k<u\le k
\end{cases}
\end{align*}
except when $k=2, (u, m)\in \{(1, 25), (2, 25), (2, 243)\}$. Observe that if $p>2(k+u)$,
then $p\geq 2(k+u)+1$. Now we take $p=P(\Delta(m, 2, k))$ in Lemma \ref{irmain} to obtain that
$G_{u+\frac{1}{2}}$ do not have a factor of degree $k$ with $k\ge 2$ except possibly
when $k=2, u=1, n=13$ or $k=2, u=2, n\in \{12, 121\}$.
We use \cite[Lemma 2.13]{stirred} with $(p, r)=(3, 1), (7, 1)$ to show that
$G_{u+\frac{1}{2}}$ do not have a factor of degree $2$ when $(u, n)=(1, 13), (2, 12)$
and $(u, n)=(2, 121)$, respectively.
\qed

\end{document}